\documentclass[12pt,a4paper]{amsart}

\newtheorem{theo+}              {Theorem}           [section]
\newtheorem{prop+}  [theo+]     {Proposition}
\newtheorem{coro+}  [theo+]     {Corollary}
\newtheorem{lemm+}  [theo+]     {Lemma}
\newtheorem{exam+}  [theo+]     {Example}
\newtheorem{rema+}  [theo+]     {Remark}
\newtheorem{defi+}  [theo+]     {Definition}
\def \r{\mbox{${\mathbb R}$}}

\newenvironment{theorem}{\begin{theo+}}{\end{theo+}}
\newenvironment{proposition}{\begin{prop+}}{\end{prop+}}
\newenvironment{corollary}{\begin{coro+}}{\end{coro+}}
\newenvironment{lemma}{\begin{lemm+}}{\end{lemm+}}
\newenvironment{definition}{\begin{defi+}}{\end{defi+}}

\usepackage{amsthm}
\theoremstyle{plain} \theoremstyle{remark}
\newtheorem{remark}{Remark}
\newtheorem{example}{Example}

\newtheorem*{ack}{\bf Acknowledgments}

\def\E{/\kern-1.0em \equiv }

\title{On $f$-harmonic morphisms between Riemannian manifolds}
\author{Ye-Lin Ou $^{*}$}
\thanks{$^{*}$Supported by
NSF of Guangxi (P. R. China), 2011GXNSFA018127.}
\address{School of Mathematics and Computer Science,\newline\indent Guangxi
University for Nationalities,
\newline\indent 188 East Daxue Road,\newline\indent Nanning, Guangxi 530006,\newline\indent P. R. China
\newline\indent E-mail:yelinou@hotmail.com}
\begin{document}

\title[$f$-Harmonic morphisms]{On $f$-harmonic morphisms between Riemannian manifolds}

\date{03/28/2011}
\maketitle

\section*{Abstract}
\begin{quote}
{\footnotesize $f$-Harmonic maps were first introduced and studied
by Lichnerowicz in \cite{Li} (see also Section 10.20 in
Eells-Lemaire's report \cite{EL}). In this paper, we study a
subclass of $f$-harmonic maps called $f$-harmonic morphisms which
pull back local harmonic functions to local $f$-harmonic functions.
We prove that a map between Riemannian manifolds is an $f$-harmonic
morphism if and only if it is a horizontally weakly conformal
$f$-harmonic map. This generalizes the well-known Fuglede-Ishihara
characterization for harmonic morphisms. Some properties and many
examples as well as some non-existence of $f$-harmonic morphisms are
given. We also study the $f$-harmonicity of conformal immersions.}
\end{quote}

\section{$f$-Harmonic maps vs. $F$-harmonic maps}
\noindent {\bf 1.1 $f$-harmonic maps} \\
Let $f:(M,g)\longrightarrow (0,\infty)$ be a smooth function. An
{\em $f$-harmonic map} is a map $\phi: (M^m,g)\longrightarrow
(N^n,h)$ between Riemannian manifolds such that $\phi|_\Omega$ is a
critical point of the $f$-energy (see \cite{Li}, and \cite{EL},
Section 10.20)
\begin{equation}\notag
E_f(\phi)=\frac{1}{2}\int_\Omega f\,|{\rm d}\phi|^2dv_g,
\end{equation}
for every compact domain $\Omega\subseteq M$. The Euler-Lagrange
equation gives the $f$-harmonic map equation ( \cite{Co1},
\cite{OND})
\begin{equation}\label{fhm}
\tau_f(\phi)\equiv f\tau(\phi)+{\rm d}\phi({\rm grad}\,f)=0,
\end{equation}
where $\tau(\phi)={\rm Tr}_g\nabla\,d \phi$ is the tension field of
$\phi$ vanishing of which means $\phi$ is a harmonic map.
\begin{example}
Let $\varphi, \;\psi,\; \phi:\r^3\longrightarrow \r^2$ be defined as
\begin{eqnarray}\notag
\varphi(x, y, z)&=& (x, y),\\\notag \psi(x, y, z)&=& (3x,
xy),\;\;\;{\rm and}\\\notag \phi(x,y,z)&=&(x, y+z).
\end{eqnarray}
Then, one can easily check that both $\varphi$ and $\psi$ are
$f$-harmonic map with $f=e^z$, $\varphi$ is a horizontally conformal
submersion whilst $\psi$ is not. Also, $\phi$ is an $f$-harmonic map
with $f=e^{y-z}$, which is a submersion but not horizontally weakly
conformal.
\end{example}
\noindent{\bf 1.2 F-harmonic map}\\
Let $F:[0,+\infty)\longrightarrow {[}0,{+}\infty)$ be a
$C^2$-function, strictly increasing on $(0,+\infty)$, and let
$\varphi: (M,g)\longrightarrow (N,h)$ be a smooth map between
 Riemannian manifolds. Then $\varphi$ is said to be  an {\em
$F$-harmonic map} if $\varphi|_\Omega$ is a critical point of the
$F$-energy functional
\begin{equation}\notag
E_F(\varphi)=\int_\Omega F(\frac{|d\varphi|^2}{2})\,v_g,
\end{equation}
for every compact domain $\Omega\subseteq M$. The equation of
$F$-harmonic maps is given by (\cite{Ar})
\begin{equation}\label{Fhm}
\tau_F(\varphi)\equiv
F'(\frac{|d\varphi|^2}{2})\tau(\varphi)+\varphi_{*}\left({\rm
grad}\,F'(\frac{|d\varphi|^2}{2})\right)=0,
\end{equation}
where $\tau(\varphi)$ denotes the tension field of $\varphi$.\\

{\em Harmonic maps}, {\em $p$-harmonic maps}, and {\em exponential
harmonic maps} are examples of $F$-harmonic maps with $F(t) = t$,
$F(t) =\frac{1}{p}(2t)^{p/2} \;(p > 4)$, and
$F(t)= e^t$ respectively (\cite{Ar}).\\

In particular, $p$-harmonic map equation can be written as
\begin{equation}\label{PT}
\tau_p(\varphi)= {\left|d \varphi \right|}^{p-2}\tau(\varphi)+d
\varphi ({\rm grad}{\left|d \varphi \right|}^{p-2})=0,
\end{equation}
{\bf 1.3 Relationship between $f$-harmonic and $F$-harmonic maps}\\
We can see from Equation (\ref{fhm}) that an $f$-harmonic map with
$f={\rm constant} >0$ is nothing but a harmonic map so both
$f$-harmonic maps and $F$-harmonic maps are generalizations of
harmonic maps. Though we were warned in \cite{Co1} that $f$-harmonic
maps should not be confused with $F$-harmonic maps and $p$-harmonic
maps, we observe that, apart from critical points, any $F$-harmonic
map is a special $f$-harmonic maps. More precisely we have
\begin{corollary}\label{C1}
Any $F$-harmonic map $\varphi: (M,g)\longrightarrow (N,h)$ without
critical points, i.e., $|d\varphi_x|\ne 0$ for all $x\in M$, is an
$f$-harmonic map with $f=F'(\frac{|d\varphi|^2}{2})$. In particular,
a $p$-harmonic map without critical points is an $f$-harmonic map
with $f=|d\varphi|^{p-2}$
\end{corollary}
\begin{proof}
Since $F$ is a $C^2$-function and strictly increasing on
$(0,+\infty)$ we have $F'(t)>0$ on $(0,+\infty)$. If the
$F$-harmonic map $\varphi: (M,g)\longrightarrow (N,h)$ has no
critical points, i.e., $|d\varphi_x|\ne 0$ for all $x\in M$, then
the function $f:(M,g)\longrightarrow (0,+\infty)$ with
$f=F'(\frac{|d\varphi|^2}{2})$ is a smooth and we see from Equations
(\ref{Fhm}) and (\ref{fhm}) that the $F$-harmonic map $\varphi$ is
an $f$-harmonic map with $f=F'(\frac{|d\varphi|^2}{2})$. The second
statement follows from the fact that for a $p$-harmonic map, $F(t)
=\frac{1}{p}(2t)^{p/2} $ and hence
$f=F'(\frac{|d\varphi|^2}{2})=|d\varphi|^{p-2}$.

\end{proof}

Another relationship between $f$-harmonic maps and harmonic maps can
be characterized as follows.
\begin{corollary}\label{C2}
A map $\phi: (M^m,g)\longrightarrow (N^n,h)$ is $f$-harmonic if and
only if $\phi: (M^m,f^{\frac{2}{m-2}}g)\longrightarrow (N^n,h)$ is a
harmonic map.
\end{corollary}
\begin{proof}
The statement that the $f$-harmoninicity of $\phi:
(M^m,g)\longrightarrow (N^n,h)$ implies the harmonicity of $\phi:
(M^m,f^{\frac{2}{m-2}}g)\longrightarrow (N^n,h)$ was stated and
proved in \cite{Li} (see also Section 10.20 in \cite{EL}). It is not
difficult to see that the converse is also true. In fact, let
$\bar{g}=f^{\frac{2}{m-2}}g$, a straightforward computation shows
that
\begin{eqnarray}\notag
 \tau(\varphi,{\bar g}) =
f^{\frac{-m}{m-2}}\tau_f (\varphi, g),
\end{eqnarray}
which completes the proof of the corollary.
\end{proof}
\noindent{\bf 1.4 A physical motivation for the study of
$f$-harmonic maps:} In physics, the equation of motion of a
continuous system of spins with inhomogeneous neighbor Heisenberg
interaction (such a model is called the inhomogeneous Heisenberg
ferromagnet) is given by
\begin{equation}\label{IHHS}
\frac{\partial u}{\partial t}=f(x)(u\times \Delta
u)+\nabla\,f\cdot(u\times \nabla\,u),
\end{equation}
where $\Omega\subseteq\r^m$ is a smooth domain in the Euclidean
space, $f$ is a real-valued function defined on $\Omega$, $u(x, t)
\in S^2$, $\times$ denotes the cross products in $\r^3$ and $\Delta$
is the Laplace operator on $\r^m$. Physically, the function $f$ is
called the coupling function, and is the continuum limit of the
coupling constants between the neighboring spins. Since $u$ is a map
into $S^2$ it is well known that the tension field of $u$ can be
written as $\tau (u)=\Delta u + |\nabla\,u|^2u$, and one can easily
check that the right hand side of the inhomogeneous Heisenberg spin
system (\ref{IHHS}) can be written as $ u \times (f\tau
(u)+\nabla\,f \cdot \nabla\,u)$. It follows that $u$ is a smooth
stationary solution of (\ref{IHHS}) if and only if $f\tau (u) +
\nabla\,f \cdot \nabla\,u = 0$, i.e., $u$ is an $f$-harmonic map. So
there is a one-to-one correspondence between the set of the
stationary solutions of the inhomogeneous Heisenberg spin system
(\ref{IHHS}) on the domain $\Omega$ and the set of $f$-harmonic maps
from $\Omega$ into $2$-sphere. The above inhomogeneous Heisenberg
spin system (\ref{IHHS}) is also called inhomogeneous
Landau-Lifshitz system (see, e.g., \cite{CSW}, \cite{CGS},
\cite{DPL}, \cite{HT}, \cite{LB}, \cite{LW}, for more details).\\

Using Corollary \ref{C2} we have the following example which
provides many stationary solutions of the inhomogeneous Heisenberg
spin system defined on $\r^3$.
\begin{example}\label{Ex2}
 $u:(\r^3, ds_0)\longrightarrow (N^{n},h)$ is an $f$-harmonic map if and only
if \\$u:(\r^3, f^2ds_{0})\longrightarrow (N^{n},h)$ is a harmonic
map. In particular, there is a 1-1 correspondence between harmonic
maps from 3-sphere $S^3\setminus\{N\}\equiv (\r^3,
\frac{4ds_0}{(1+|x|^2)^2})\longrightarrow (N^{n},h)$ and
$f$-harmonic maps with $f=\frac{2}{1+|x|^2}$ from Euclidean 3-space
$\r^3\longrightarrow (N^{n},h)$. When $(N^n,h)=S^2$, we have a $1-1$
correspondence between the set of harmonic maps $S^3\longrightarrow
S^2$ and the set of stationary solutions of the inhomogeneous
Heisenberg spin system on $\r^3$. Similarly, there is a 1-1
correspondence between harmonic maps from hyperbolic 3-space
$H^3\equiv (D^3, \frac{4ds_0}{(1-|x|^2)^2})\longrightarrow (N^n,h)$
and $f$-harmonic maps $(D^3,ds_0)\longrightarrow (N^n,h)$ with
$f=\frac{2}{1-|x|^2}$ from unit disk in Euclidean 3-space.
\end{example}
\noindent{\bf 1.5 A little more about $f$-harmonic maps}
\begin{corollary}\label{2f}
If $\phi: (M^m,g)\longrightarrow (N^n,h)$ is an $f_1$-harmonic map
and also an $f_2$-harmonic map, then ${\rm grad}(f_1/f_2 )\in
\ker\,{\rm d} \phi$.
\end{corollary}
\begin{proof}
This follows from
\begin{eqnarray}\notag
\tau_{f_1}(\phi)\equiv f_1\tau(\phi)+{\rm d}\phi({\rm
grad}\,f_1)=0,\\\notag \tau_{f_2}(\phi)\equiv f_2\tau(\phi)+{\rm
d}\phi({\rm grad}\,f_2)=0,
\end{eqnarray}
 and hence
\begin{eqnarray}\notag
{\rm d}\phi({\rm grad\,ln}\,(f_1/f_2))=0.
\end{eqnarray}
\end{proof}

\begin{proposition}
 A conformal immersion $\phi: (M^m,g)\longrightarrow (N^{n},h)$
with $\phi^{*}h=\lambda^{2}g$ is $f$-harmonic if and only if it is
$m$-harmonic and $f=C\lambda^{m-2}$. In particular, an isometric
immersion is $f$-harmonic if and only if $f=const$ and hence it is
harmonic.
\end{proposition}
\begin{proof}
It is not difficult to check (see also \cite{Ta}) that for a
conformal immersion $\phi: (M^m,g)\longrightarrow (N^{n},h)$ with
$\phi^{*}h=\lambda^{2}g$, the tension field is given by
\begin{eqnarray}\label{CTau}\notag
\tau(\phi)=m\lambda^2 \eta+(2-m){\rm d}\phi \left( {\rm grad}\, {\rm
ln} \lambda\right),
\end{eqnarray}
so we can compute the $f$-tension field to have
\begin{eqnarray}\label{fTau}\notag
\tau_f(\phi)=f[m\lambda^2 \eta+{\rm d}\phi \left( {\rm grad}\, {\rm
ln} (\lambda^{2-m}\,f)\right)],
\end{eqnarray}
where $\eta$ is the mean curvature vector of the submanifold $\phi
(M)\subset N$. Noting that $\eta$ is normal part whilst ${\rm d}\phi
\left( {\rm grad}\, {\rm ln} \lambda^{2-m}\,f\right)$ is the
tangential part of $\tau_f(\phi)$ we conclude that $\tau_f(\phi)=0$
if and only if
\begin{equation}
\begin{cases} m\lambda^2 \eta=0,\\\notag {\rm d}\phi \left(
{\rm grad}\, {\rm ln}( \lambda^{2-m}\,f)\right)=0.
\end{cases}
\end{equation}
It follows that $\eta=0$ and ${\rm grad}\,( {\rm ln}
(\lambda^{2-m}\,f))=0$ since $\phi$ is an immersion. From these we
see that $\phi$ is a minimal conformal immersion which means it is
an $m$-harmonic map (\cite{Ta}) and that $f=C\lambda^{m-2}$. Thus,
we obtain the first statement. The second statement follows from the
first one with $\lambda=1$.
\end{proof}

\section{$f$-Harmonic morphisms}
\indent A {\em horizontally weakly conformal} map is a map
$\varphi:(M, g)\longrightarrow (N, h)$ between Riemannian manifolds
such that for each $x \in M$ at which ${\rm d} \varphi_{x} \neq 0$,
the restriction ${\rm d} \varphi_{x}|_{H_x} : H_x \longrightarrow
T_{\varphi(x)} N$ is conformal and surjective, where the horizontal
subspace $H_{x}$ is the orthogonal complement of $V_x = {\rm ker}\,
{\rm d} \varphi_x$ in $T_{x}M$. It is not difficult to see that
there is a number $\lambda (x) \in (0,\infty)$ such that $h({\rm d}
\varphi (X),{\rm d} \varphi (Y)) = \lambda^{2}(x)g(X,Y)$ for any $X,
Y \in H_x$. At the point $x \in M$ where ${\rm d} \varphi_{x} = 0$
one can let $\lambda(x) = 0$ and obtain a continuous function
$\lambda : M \longrightarrow R$ which is called the {\em dilation}
of a horizontally weakly conformal map $\varphi$. A non-constant
horizontally weakly conformal map $\varphi$ is called {\em
horizontally homothetic} if the gradient of $\lambda^2(x) $ is
vertical meaning that $ X(\lambda^{2}) \equiv 0$ for any horizontal
vector field $X$ on $M$. Recall that a $C^{2}$ map $\varphi:(M,
g)\longrightarrow (N, h)$ is a {\em $p$-harmonic morphism} $(p > 1)$
if it preserves the solutions of $p$-Laplace equation in the sense
that for any $p$-harmonic function $f:U \longrightarrow \mathbb{R}$,
defined on an open subset U of $N$ with ${\varphi}^{-1}(U)$
non-empty, $f\circ\varphi :{\varphi}^{-1}(U) \longrightarrow
\mathbb{R} $ is a $p$-harmonic function. A $p$-harmonic morphism can
be characterized as a horizontally weakly conformal $p$-harmonic map
(see \cite{Fu}, \cite{Is}, \cite{Lo} for details).
\begin{definition}
Let $f : (M, g) \longrightarrow (0, \infty)$ be a smooth function. A
$C^2$-function $u : U \longrightarrow \r$ defined on an open subset
$U$ of $M$ is called $f$-harmonic if
\begin{equation}\label{fLa}
\Delta_f^M\, u\equiv f\Delta^M\, u + g({\rm grad}f, {\rm grad}\,u) =
0.
\end{equation}
A continuous map $\phi : (M^m, g) \longrightarrow (N^n, h)$ is
called an $f$-harmonic morphism if for every harmonic function $u$
defined on an open subset $V$ of $N$ such that $\phi^{-1}(V)$ is
non-empty, the composition $u \circ \phi$ is $f$-harmonic on
$\phi^{-1}(V)$.
\end{definition}

\begin{theorem}\label{cha}
Let $\phi : (M^m, g) \longrightarrow (N^n, h)$ be a smooth map.
Then, the following are equivalent:\\
$(1)$ $\phi$ is an $f$-harmonic morphism;\\
$(2)$ $\phi$ is a horizontally weakly conformal $f$-harmonic map;\\
$(3)$ There exists a smooth function $\lambda ^2$ on $M$ such that
\begin{eqnarray}\notag
\Delta_f^{M}(u\circ \phi) &=&f\lambda^2(\Delta^N\,u)\circ\phi
\end{eqnarray}
for any $C^2$-function $u$ defined on (an open subset of) $N$.
\end{theorem}
\begin{proof}
We will need the following lemma to prove the theorem.
\begin{lemma}\label{IsL}$($\cite{Is}$)$ For any point $q \in (N^n,h)$ and any constants
$C_\sigma ,\; C_{\alpha\beta}$ with $C_{\alpha\beta} =
C_{\beta\alpha}$ and $\sum_{\alpha=1}^n\,C_{\alpha\alpha} = O$,
there exists a harmonic function $u$ on a neighborhood of $q$ such
that $u_\sigma (q)= C_\sigma , u_{\alpha\beta}(q)= C_{\alpha\beta}$.
\end{lemma}
Let $\phi : (M^m, g) \longrightarrow (N^n, h)$ be a map and let
$p\in M$. Suppose that $\phi(x)=(\phi^1(x),\phi^2(x),\cdots,
\phi^n(x))$ is the local expression of $\phi$ with respect to the
local coordinates $\{x^i\}$ in the neighborhood $\phi^{-1}(V)$ of
$p$ and $\{y^\alpha\}$ in a neighborhood $V$ of $q=\phi(p)\in N$.
Let $u : V \longrightarrow \r$ defined on an open subset $V$ of $N$.
Then, a straightforward computation gives
\begin{eqnarray}\notag
\Delta_f^{M}(u\circ \phi)&=&f\Delta^{M}(u\circ \phi)+d(u\circ
\phi)({\rm grad} f)\\\notag &=&fu_{\alpha\beta}g({\rm
grad}\phi^{\alpha}, {\rm
grad}\phi^{\beta})+fu_{\alpha}\Delta^M\phi^{\alpha}+d(u\circ
\phi)({\rm grad} f)\\\label{TE0} &=&f g({\rm grad}\phi^{\alpha},
{\rm
grad}\phi^{\beta})u_{\alpha\beta}+[f\,\Delta^M\phi^{\sigma}+({\rm
grad} f )\phi^{\sigma}]u_\sigma.
\end{eqnarray}
By Lemma \ref{IsL}, we can choose a local harmonic function $u$ on
$V\subset N$ such that $u_\sigma (q)= C_\sigma=0 \;\forall\;
\sigma=1, 2, \cdots, n,\; u_{\alpha\beta}(q)=1\;\; (\alpha\ne
\beta), {\rm and\; all\; other}\;
u_{\rho\sigma}(q)=C_{\rho\sigma}=0$ and substitute it into
(\ref{TE0}) to have
\begin{eqnarray}\label{TE2}
g({\rm grad}\phi^{\alpha}, {\rm
grad}\phi^{\beta})=0,\;\forall\;\alpha\ne \beta=1, 2, \cdots, n.
\end{eqnarray}
Note that the choice of such functions implies
\begin{eqnarray}\label{Hd}
h^{\alpha\beta}(\phi(p))=0,\;\forall\;\alpha\ne \beta=1, 2, \cdots,
n.
\end{eqnarray}
Another choice of harmonic function $u$ with $C_{11}=1,
\;C_{\alpha\alpha}=-1\;\; (\alpha\ne 1)$ and all other $C_\sigma,
C_{\alpha\beta}=0$ for Equation (\ref{TE0}) gives
\begin{eqnarray}\label{TE3}
g({\rm grad}\phi^{1}, {\rm grad}\phi^{1})-g({\rm grad}\phi^{\alpha},
{\rm grad}\phi^{\alpha})=0,\;\forall\;\alpha\ne \beta= 2, 3, \cdots,
n.
\end{eqnarray}
Note also that for these choices of harmonic functions $u$ we have
\begin{eqnarray}\label{Hs}
h^{11}(\phi(p))-h^{\alpha\alpha}(\phi(p))=0,\;\forall\;\alpha\ne
\beta= 2, 3, \cdots, n.
\end{eqnarray}
It follows from (\ref{TE2}), (\ref{Hd}), (\ref{TE3}) and (\ref{Hs})
that the $f$-harmonic morphism $\phi$ is a horizontally weakly
conformal map
\begin{equation}\label{HWC}
g({\rm grad}\phi^{\alpha}, {\rm
grad}\phi^{\beta})=\lambda^2h^{\alpha\beta}\circ\phi.
\end{equation}
Substituting horizontal conformality equation (\ref{HWC}) into
(\ref{TE0}) we have
\begin{eqnarray}\notag
\Delta_f^{M}(u\circ \phi)&=& f\lambda^2
(h^{\alpha\beta}\circ\phi)u_{\alpha\beta}+[f\,\Delta^M\phi^{\sigma}+({\rm
grad} f )\phi^{\sigma}]u_\sigma\\\notag
&=&f\lambda^2(\Delta^N\,u)\circ\phi+[f\,\Delta^M\phi^{\sigma}+f\lambda^2(h^{\alpha\beta}{\bar\Gamma}_{\alpha\beta}^\sigma)\circ\phi+({\rm
grad} f )\phi^{\sigma}]u_\sigma\\\label{Comp}
&=&f\lambda^2(\Delta^N\,u)\circ\phi+\,{\rm d}u\,(\tau_f(\phi))
\end{eqnarray}
for any function $u$ defined (locally) on $N$. By special choice of
harmonic function $u$ we conclude that the $f$-harmonic morphism is
an $f$-harmonic map. Thus, we obtain the implication
$``(1)\Longrightarrow(2)"$. Note that the only assumption we used to
obtain Equation (\ref{Comp}) is the horizontal conformality
(\ref{HWC}). Therefore, it follows from (\ref{Comp}) that
$``(2)\Longrightarrow(3)"$. Finally, $``(3)\Longrightarrow(1)"$ is
clearly true. Thus, we complete the proof of the theorem.
\end{proof}
Similar to harmonic morphisms we have the following regularity
result.
\begin{corollary}
For $m\ge 3$, an $f$-harmonic morphism $\phi : (M^m, g)
\longrightarrow (N^n, h)$ is smooth.
\end{corollary}
\begin{proof}
In fact, by Corollary \ref{C1}, if $m \ne 2$ and $\phi : (M^m, g)
\longrightarrow (N^n, h)$ is an $f$-harmonic morphism, then $\phi :
(M^m, f ^{2/(m-2)}g) \longrightarrow (N^n, h)$ is a harmonic map
 and hence a harmonic morphism, which
is known to be smooth (see, e. g., \cite{BW}).
\end{proof}
It is well known that the composition of harmonic morphisms is again
a harmonic morphism. The composition law for $f$-harmonic morphisms,
however, will need to be modified accordingly. In fact, by the
definitions of harmonic morphisms and $f$-harmonic morphisms we have
\begin{corollary}
Let $\phi : (M^m, g) \longrightarrow (N^n, h)$ be an $f$-harmonic
morphism with dilation $\lambda_1$ and $\psi : (N^n, h)
\longrightarrow (Q^l, k)$ a harmonic morphism with dilation
$\lambda_2$. Then the composition $\psi\circ\phi : (M^m, g)
\longrightarrow (Q^l, k)$ is an $f$-harmonic morphism with dilation
$\lambda_1(\lambda_2\circ \phi)$.
\end{corollary}
More generally, we can prove that $f$-harmonic morphisms pull back
harmonic maps to $f$-harmonic maps.
\begin{proposition}\label{P1}
Let $\phi : (M^m, g) \longrightarrow (N^n, h)$ be an $f$-harmonic
morphism with dilation $\lambda$ and $\psi : (N^n, h)
\longrightarrow (Q^l, k)$ a harmonic map. Then the composition
$\psi\circ\phi : (M^m, g) \longrightarrow (Q^l, k)$ is an
$f$-harmonic map.
\end{proposition}
\begin{proof}
It is well known (see e.g., \cite{BW}, Proposition 3.3.12) that the
tension field of the composition map is given by
\begin{align}\notag
\tau(\psi \circ \phi)=&{\rm d}\psi (\tau (\phi))+{\rm Tr}_{g}\nabla
{\rm d}\psi ( {\rm d} \phi,{\rm d} \phi),
\end{align}
from which we have the $f$-tension of the composition $\psi\circ
\phi$ given by
\begin{align}\label{ftc}
\tau_f(\psi \circ \phi)=&{\rm d}\psi (\tau_f (\phi))+f{\rm
Tr}_{g}\nabla {\rm d}\psi ( {\rm d} \phi,{\rm d} \phi).
\end{align}
Since $\phi$ is an $f$-harmonic morphism and hence a horizontally
weakly conformal $f$-harmonic map with dilation $\lambda$, we can
choose a local orthonormal frame $\{e_1,\ldots, e_n, e_{n+1},
\ldots,e_m\}$ around $p\in M$ and $\{\epsilon_1,\ldots,
\epsilon_n\}$ around $\phi(p)\in N$ so that
\begin{equation}\notag
\begin{cases}
{\rm d} \phi(e_i)=\lambda\epsilon_i,\;\;\;i=1, \ldots, n,\\
{\rm d} \phi(e_\alpha)=0, \;\;\;\alpha={n+1}, \ldots, m.
\end{cases}
\end{equation}
Using these local frames we compute
\begin{eqnarray}\notag
{\rm Tr}_{g}\nabla {\rm d}\psi ( {\rm d} \phi,{\rm d}
\phi)&=&\sum_{i=1}^m\nabla {\rm d}\psi ( {\rm d} \phi e_i,{\rm d}
\phi e_i)=\lambda^2\left(\sum_{i=1}^n\nabla {\rm d}\psi
(\epsilon_i,\epsilon_i)\right)\circ\phi\\\notag
&=&\lambda^2\tau(\psi)\circ\phi.
\end{eqnarray}
Substituting this into (\ref{ftc}) we have
\begin{align}\label{ftc2}\notag
\tau_f(\psi \circ \phi)=&f{\rm d}\psi (\tau
(\phi))+f\lambda^2\tau(\psi)\circ\phi+{\rm d}(\psi\circ\phi)({\rm
grad}\,f)\\\notag =&{\rm d}\psi (\tau_f
(\phi))+f\lambda^2\tau(\psi)\circ\phi,
\end{align}
from which the proposition follows.
\end{proof}

\begin{theorem}\label{2-1}
Let $\phi: (M^m,g)\longrightarrow (N^{n},h)$ be a horizontally
weakly conformal map with
$\varphi^{*}h=\lambda^{2}g|_{\mathcal{H}}$. Then, any two of the
following conditions imply the other one.
\begin{itemize}
\item[(1)] $\phi$ is an $f$-harmonic map and hence an $f$-harmonic morphism;
\item[(2)] ${\rm grad}(\,f\lambda^{2-n})$ is vertical;
\item[(3)] $\phi$ has minimal fibers.
\end{itemize}
\end{theorem}
\begin{proof}
It can be check (see e.g., \cite{BW}) that the tension field of a
horizontally weakly conformal map $\phi: (M^m,g)\longrightarrow
(N^{n},h)$  is given by
\begin{equation}\notag
\tau(\phi)=-(m-n)d\phi(\mu)+(2-n)d \phi({\rm grad}\ln \lambda),
\end{equation}
where $\lambda$ is the dilation of the horizontally weakly conformal
map $\phi$ and $\mu$ is the mean curvature vector field of the
fibers. It follows that the $f$-tension field of $\phi$ can be
written as
\begin{equation}\notag
\tau_f(\phi)=-(m-n)fd\phi(\mu)+fd \phi({\rm grad}\ln
\lambda^{2-n})+d \phi({\rm grad}\,f),
\end{equation}
or, equivalently,
 \begin{equation}\notag
\tau_f(\phi)=f[-(m-n)d\phi(\mu)+d \phi({\rm grad}\ln
(f\lambda^{2-n}))]=0.
\end{equation}
From this we obtain the theorem.
\end{proof}
An immediate consequence is the following
\begin{corollary}\label{Riem}
$(a)$ A horizontally homothetic map (in particular, a Riemannian
submersion) $\phi: (M^m,g)\longrightarrow (N^{n},h)$ is an
$f$-harmonic morphism if and only if $-(m-n)\mu+{\rm grad}\ln f$
is vertical;\\
$(b)$ A weakly conformal map $\phi: (M^m,g)\longrightarrow
(N^{m},h)$ with conformal factor $\lambda$ of same dimension spaces
is $f$-harmonic and hence an $f$-harmonic morphism if and only if
$f=C\lambda^{m-2}$ for some constant $C>0$;\\
$(c)$ An horizontally weakly conformal map $\phi:
(M^m,g)\longrightarrow (N^2,h)$ is an $f$-harmonic map and hence an
$f$-harmonic morphism if and only if $-(m-2)\mu+{\rm grad}\ln f$ is
vertical.
\end{corollary}
Using the characterizations of $f$-harmonic morphisms and
$p$-harmonic morphisms and Corollary \ref{C1} we have the following
corollary which provides many examples of $f$-harmonic morphisms.
\begin{corollary}\label{pfhm}
A map $\phi: (M^m,g)\longrightarrow (N^{n},h)$ between Riemannian
manifolds is a $p$-harmonic morphism without critical points if and
only if it is an $f$-harmonic morphism with $f=|\rm d \phi|^{p-2}$.
\end{corollary}
\begin{example}
M$\ddot{\rm o}$bius transformation
$\phi:\r^{m}\setminus\{0\}\longrightarrow \r^{m}\setminus\{0\}$
defined by
\begin{equation}\notag
\phi(x)=a+\frac{r^2}{|x-a|^2}(x-a)
\end{equation}
is an $f$-harmonic morphism with $f(x)=C(\frac{r}{|x-a|})^{2(m-2)}$.
In fact, it is well known that the M$\ddot{\rm o}$bius
transformation is a conformal map between the same dimensional
spaces with the dilation $\lambda= \frac{r^2}{|x-a|^2}$.  It follows
from \cite{MV} that $\phi$ is an $m$-harmonic morphism, and hence by
Corollary \ref{pfhm}, the inversion is an $f$-harmonic morphism with
$f=|\rm d
\phi|^{m-2}=(\sqrt{m}\lambda)^{m-2}=C(\frac{r}{|x-a|})^{2(m-2)}$.
\end{example}
The next example is an $f$-harmonic morphism that does not come from
a $p$-harmonic morphism,
\begin{example}
The map from Euclidean $3$-space into hyperbolic plane $\phi:
(\r\times \r\times \r^+, ds_0^2)\longrightarrow H^2\equiv(\r\times
\{0\}\times \r^+,\frac{1}{z^2}ds_0^2)$ with
$\phi(x,y,z)=(x,0,\sqrt{y^2+z^2\,})$ is an $f$-harmonic morphism
with $f=1/z$. Similarly, we know from \cite{Gu} that the map $\phi:
H^3\equiv(\r\times \r\times \r^+,\frac{1}{z^2}ds_0^2)\longrightarrow
H^2\equiv(\r\times \{0\}\times \r^+,\frac{1}{z^2}ds_0^2)$ with
$\phi(x,y,z)=(x,0,\sqrt{y^2+z^2\,})$ is a harmonic morphism. It
follows from Example \ref{Ex2} that the map from Euclidean space
into hyperbolic plane $\phi: (\r\times \r\times \r^+,
ds_0^2)\longrightarrow H^2\equiv(\r\times \{0\}\times
\r^+,\frac{1}{z^2}ds_0^2)$ with $\phi(x,y,z)=(x,0,\sqrt{y^2+z^2\,})$
is an $f$-harmonic map with $f=1/z$. Since this map is also
horizontally weakly conformal it is an $f$-harmonic morphism by
Theorem \ref{cha}.
\end{example}
\begin{example}\label{E4}
Any harmonic morphism $\phi: (M^m,g)\longrightarrow (N^{n},h)$ is an
$f$-harmonic morphism for a positive function $f$ on $M$ with
vertical gradient, i.e., $d\phi({\rm grad} f)=0$. In particular, the
radial projection $\phi:\r^{m+1}\setminus\{0\}\longrightarrow S^m,
\;\;\phi(x)=\frac{x}{|x|}$ is an $f$-harmonic morphism for
$f=\alpha(|x|)$, where $\alpha:(0,\infty)\longrightarrow (0,\infty)$
is any smooth function. In fact, we know from \cite{BW} that the
radial projection is a harmonic morphisms and on the other hand, on
can check that the function $f=\alpha(|x|)$ is positive and has
vertical gradient.
\end{example}
Using the property of $f$-harmonic morphisms and Sacks-Uhlenberg's
well-known result on the existence of harmonic $2$-spheres we have
the following proposition which gives many examples of $f$-harmonic
maps from Euclidean $3$-space into a manifold whose universal
covering space is not contractible.

\begin{proposition}\label{P2}
For any Riemannian manifold whose universal covering space is not
contractible, there exists an $f$-harmonic map $\phi: \left(\r^3,\;
ds_0^2\right)\longrightarrow (N^n, h)$ from Euclidean $3$-space with
$f(x)=\frac{2}{1+|x|^2}$.
\end{proposition}
\begin{proof}
Let $ds_0^2$ denote the Euclidean metric on $\r^3$. It is well known
that we can use the inverse of the stereographic projection to
identify $\left(\r^3,\; \frac{4ds_0^2}{(1+|x|^2)^2}\right)$ with
 $S^3\setminus
\{N\}=\{(u_1,u_2,u_3,u_4)|\sum_{i=1}^4\,u_i^2=1, \; u_4\ne 1\}$, the
Euclidean $3$-sphere minus the north pole. In fact, the
identification is given by the isometry
\begin{equation}\notag
\sigma: \left(\r^3,\;
\frac{4ds_0^2}{(1+|x|^2)^2}\right)\longrightarrow S^3\setminus
\{N\}\subseteq \r^4
\end{equation}
with $\sigma(x_1,x_2, x_3)=(\frac{2x_{1}}{1+|x|^2},\;
\frac{2x_{2}}{1+|x|^2}, \;\frac{2x_{3}}{1+|x|^2}, \;
\frac{|x|^2-1}{1+|x|^2})$. One can check that under this
identification, the Hopf fiberation $\phi: \left(\r^3,\;
\frac{4ds_0^2}{(1+|x|^2)^2}\right)\cong S^3\setminus
\{N\}\longrightarrow S^2$ can be written as
\begin{eqnarray}\notag
\phi(x_1,x_2, x_3)= (|z|^2-|w|^2, 2zw),
\end{eqnarray}
where $z= \frac{2x_{1}}{1+|x|^2}+i \frac{2x_{2}}{1+|x|^2}, \;\;
w=\frac{2x_{3}}{1+|x|^2}+i \frac{|x|^2-1}{1+|x|^2}$. It is well
known (see, e.g., \cite{BW}) that the Hopf fiberation $\phi$ is a
harmonic morphism with dilation $\lambda=2$. So, by Corollary
\ref{C1}, $\phi: \left(\r^3,\; ds_0^2\right)\longrightarrow S^2$ is
an $f$-harmonic map with $f=\frac{2}{1+|x|^2}$. It is easy to see
that this map is also horizontally conformal submersion and hence,
by Theorem \ref{cha}, it is an $f$-harmonic morphism. On the other
hand, by a well-known result of Sacks-Uhlenbeck's, we know that
there is a harmonic map $\rho: S^2 \longrightarrow (N^n,h)$ from
$2$-sphere into a manifold whose covering space is not contractible.
It follows from Proposition \ref{P1} that the composition $\rho\circ
\phi: \left(\r^3,\; ds_0^2\right)\longrightarrow (N^n,h)$ is an
$f$-harmonic map with $f=\frac{2}{1+|x|^2}$.
\end{proof}
\begin{remark}
We notice that the authors in \cite{Co2} and \cite{HT} used heat
flow method to study the existence of $f$-harmonic maps from closed
unit disk $D^2\longrightarrow S^2$ sending boundary to a single
point. The $f$-harmonic morphism $\phi: \left(\r^3,\;
ds_0^2\right)\longrightarrow S^2$ in Proposition \ref{P2} clearly
restrict to an $f$-harmonic map $\phi: \left(D^3,\;
ds_0^2\right)\longrightarrow S^2$ from 3-dimensional open disk into
$S^2$. It would be interesting to know if there is any $f$-harmonic
map from higher dimensional closed disk into two sphere. Though we
know that $\phi: (M^m,g)\longrightarrow (N^n,h)$ being $f$-harmonic
implies $\phi: (M^m,f^{\frac{2}{m-2}}g)\longrightarrow (N^n,h)$
being harmonic we need to be careful trying to use results from
harmonic maps theory since a conformal change of metric may change
the curvature and the completeness of the original manifold $(M^m,
g)$.
\end{remark}
As we remark in Example \ref{E4} that any harmonic morphism is an
$f$-harmonic morphism provided $f$ is positive with vertical
gradient, however, such a function need  not always exist as the
following proposition shows.
\begin{proposition}\label{PL}
A Riemannian submersion $\phi: (M^m,g)\longrightarrow (N^n, h)$ from
non-negatively curved compact manifold with minimal fibers is an
$f$-harmonic morphism if and only if $f=C>0$. In particular, there
exists no nonconstant positive function on $S^{2n+1}$ so that the
Hopf fiberation $\phi: S^{2n+1}\longrightarrow (N^n, h)$ is an
$f$-harmonic morphism.
\end{proposition}
\begin{proof}
By Corollary \ref{Riem}, a Riemannian submersion $\phi:
(M^m,g)\longrightarrow (N^n, h)$ with minimal fibers is an
$f$-harmonic morphism if and only if  ${\rm grad}\ln f$ is vertical,
i.e., ${\rm d}\phi({\rm grad\,ln}\,f)=0$. This, together with the
following lemma will complete the proof of the  proposition.\\

{\bf Lemma:} Let $\phi :(M^m,g)\longrightarrow (N^n, h)$ be any
Riemannian submersion of a compact positively curved manifold $M$.
Then, there exists no (nonconstant) function $f:M\longrightarrow
\mathbb{R}$ such that ${\rm d}\phi({\rm grad\,ln}\,f=0)$.\\

{\bf Proof of the Lemma:} Suppose $f:(M^m,g)\longrightarrow
\mathbb{R}$ has vertical gradient. Consider
\begin{equation}\notag
\left( M,e^{\varepsilon f}g\right)
\end{equation}
where $\varepsilon >0$ is a sufficiently small constant.

If $\varepsilon $ is small enough, then $e^{2\varepsilon f}g$ is
positively curved. One can check that
\begin{equation}\label{gdf}
\phi :\left( M,e^{2\varepsilon f}g\right) \longrightarrow (N, h)
\end{equation}
is a horizontally homothetic submersion with dilation $\lambda^2=
e^{-2\varepsilon f}$ since $f$ has vertical gradient. By the main
theorem in \cite{OW} we conclude that the map $\phi $ defined in
(\ref{gdf}) is a Riemannian submersion, which implies that the
dilation and hence the function $f $ has to be a constant.
\end{proof}
\begin{remark}
It would be very interesting to know if there exists any
$f$-harmonic morphism (or $f$-harmonic map) $\phi:
S^{2n+1}\longrightarrow (N^n, h)$ with non-constant $f$. Note that
for the case of $n=2$, the problem of classifying all $f$-harmonic
morphisms $\phi: (S^{3}, g_0)\longrightarrow (N^2, h)$ (where $g_0$
denotes the standard Euclidean metric on the 3-sphere) amounts to
classifying all harmonic morphisms $\phi: (S^{3},
f^2g_0)\longrightarrow (N^2, h)$ from conformally flat $3$-spheres.
A partial result on the latter problem was given in \cite{He} in
which the author proves that a submersive harmonic morphism $\phi:
(S^{3}, f^2g_0)\longrightarrow (N^2, h)$ with non-vanishing
horizontal curvature is the Hopf fiberation up to an isometry of
$(S^3, g_0)$. This implies that there exists no submersive
$f$-harmonic morphism $\phi: (S^{3}, g_0)\longrightarrow (N^2, h)$
with non-constant $f$ and the horizontal curvature
$K_{\mathcal{H}}(f^2g_0)\ne0$.
\end{remark}
\begin{proposition}
For $m > n \ge 2$, a polynomial map $($i.e. a map whose component
functions are polynomials$)$ $\phi : \r^m\longrightarrow \r^n$ is an
$f$-harmonic morphism if and only if $\phi $ is a harmonic morphism
and $f$ has vertical gradient.
\end{proposition}
\begin{proof}
Let $\phi : \r^m\longrightarrow \r^n$ be a polynomial map $($i.e. a
map whose component functions are polynomials$)$. If $\phi$ is an
$f$-harmonic morphism, then, by Theorem \ref{cha}, it is a
horizontally weakly conformal $f$-harmonic map. It was proved in
\cite{ABB} that any horizontally weakly conformal polynomial map
between Euclidean spaces has to be harmonic. This implies that
$\phi$ is also a harmonic morphism, and in this case we have ${\rm
d}\phi({\rm grad}\,f)=0$ from (\ref{fhm}).
\end{proof}
\begin{example}
$\phi:\r^3\cong\r\times \mathbb{C}\longrightarrow \mathbb{C}$ with
$\phi(t,z)=p(z)$, where $p(z)$ is any polynomial function in $z$, is
an $f$-harmonic morphism with $f(t,z)=\alpha(t)$ for any positive
smooth function $\alpha$.
\end{example}

Finally, we would like to point out that our notion of $f$-harmonic
morphisms should not be confused with $h$-harmonic morphisms studied
in \cite{Fu} and \cite{BG}, where an $h$-harmonic function is
defined to be a solution of $\Delta\, u + 2g({\rm grad}\ln h), {\rm
grad}(u)) = 0$ (or equivalently, $h\Delta\, u + 2{\rm d} u({\rm
grad}\, h) = 0$), and an $h$-harmonic morphism is a continuous map
between Riemannian manifolds which pulls back local harmonic
functions to $h$-harmonic functions.

\begin{ack}
I am very grateful to F. Wilhelm for some useful conversations
during the preparation of the paper, especially, I would like to
thank him for offering the proof of the Lemma in the proof of
Proposition \ref{PL}.
\end{ack}

\end{document}